\documentclass[12pt]{amsart}
\usepackage[margin=1.1in]{geometry}

\newtheorem{theorem}{Theorem}[section]
\newtheorem{lemma}[theorem]{Lemma}

\newtheorem{proposition}[theorem]{Proposition}
\newtheorem{remark0}[theorem]{Remark}
\newtheorem{example0}[theorem]{Example}
\newtheorem{definition}[theorem]{Definition}

\newenvironment{example}{\begin{example0}\rm}{\end{example0}}
\newenvironment{remark}{\begin{remark0}\rm}{\end{remark0}}

\newcommand{\propref}[1]{Proposition~\ref{#1}}
\newcommand{\thmref}[1]{Theorem~\ref{#1}}
\newcommand{\lemref}[1]{Lemma~\ref{#1}}

\def\max{{\mathfrak m}}                   
\newcommand{\m}{\mathfrak m}
\def\maxn{{\mathfrak n}}                   

\def\md{{{\bf 1}_d}}
\def\mr{{{\bf 1}_r}}

\def\uz{{\underline z}}

\def\HF{{\operatorname{H\!F}}}

\def\HS{\operatorname{H\!S}}

\def\reg{\operatorname{reg}}
\def\deg{\operatorname{deg}}
\def\dim{\operatorname{dim}}
\def\ann{\operatorname{Ann}}

\def\img{\operatorname{Im}}
\def\ker{\operatorname{Ker}}
\def\length{\operatorname{Length}}

\begin{document}
\title[The structure of the inverse system  of  Gorenstein k-algebras]{{\bf
The structure of the inverse system  of  Gorenstein  $k$-algebras}}

\author[J. Elias]{J. Elias ${}^{*}$}
\thanks{${}^{*}$
Partially supported by MTM2016-78881-P, Ministerio de Econom'a y Competitividad, Spain. }
\address{Juan Elias
\newline \indent Departament d'\`{A}lgebra i Geometria
\newline \indent Facultat de Matem\`{a}tiques
\newline \indent Universitat de Barcelona
\newline \indent Gran Via 585, 08007
Barcelona, Spain}  \email{{\tt elias@ub.edu}}

\author[M. E. Rossi]{M. E. Rossi ${}^{**}$}
\thanks{${}^{**}$
Partially supported by PRIN 2015EYPTSB-008 Geometry of Algebraic Varieties.  }

\address{Maria  Evelina  Rossi
\newline \indent Dipartimento di Matematica
\newline \indent Universit\`{a} di genova
\newline \indent Via Dodecaneso 35, 16146 Genova, Italy}
\email{{\tt rossim@dima.unige.it}}

\begin{abstract}
 Macaulay's Inverse System gives an effective method to construct  Artinian Gorenstein $k$-algebras.   
To date  a general  structure for  Gorenstein   $k$-algebras   of any dimension (and codimension)  is not understood. In this paper we extend Macaulay's correspondence  characterizing  the submodules of the divided power ring  in one-to-one correspondence with Gorenstein d-dimensional $k$-algebras.  We discuss effective methods for constructing Gorenstein graded rings. Several examples illustrating our results are given.

\end{abstract}

\maketitle

\centerline{\small \it In honor of Giuseppe Valla our teacher and friend} \vskip 0.7cm

\section{Introduction}

Gorenstein rings were introduced by Grothendieck, who named them because of their relation to a duality property of singular plane curves studied by Gorenstein, \cite{Gor52}, \cite{Gro55}.
The zero-dimensional case had previously been  studied by  Macaulay, \cite{mac16}.
Gorenstein rings are very common and significant in many areas of mathematics, as it can be seen in Bass's paper  \cite{Bas63}, see also \cite{Hun99}.
They have appeared as an important component in a significant number of problems and have proven useful in a wide variety of applications
in commutative algebra,  singularity theory, number theory  and more recently in combinatorics, among other areas.

\smallskip
Gorenstein rings are a generalization of complete intersections, and indeed the two notions coincide in codimension two.
Codimension three Gorenstein rings are completely described   by  Buchsbaum and Eisenbud's structure theorem, \cite{BE77}.
More recently Reid in \cite{Rei13}  studied the
projective resolution of Gorenstein ideals of codimension $4$, aiming to extend the result of Buchsbaum and Eisenbud.
Kustin and Miller in a series of papers studied the structure of Gorenstein ideal of higher codimension, see \cite{KM83} and the references therein.

Notice that the lack of a general structure of homogeneous  Gorenstein ideals of higher codimension   is the
main obstacle to extending the Gorenstein liaison theory in codimension at least three;
the codimension two Gorenstein liaision case is well understood,  see \cite{KMMNP01}.
See, for instance,
\cite{MP97}, \cite{KM83} and \cite{IS05} for some constructions of particular families of  Gorenstein algebras.

\bigskip
Let $k$ be a   field  and let $I$ be an ideal (not necessarily homogeneous) of the power series ring $R$  (or of the polynomial ring in the homogeneous case).
As an effective consequence of Matlis duality,   it is known that an Artinian ring $R/I $ is a Gorenstein $k$-algebra if and only if $I$ is the ideal of a system of polynomial differential operators with constant coefficients having a unique solution.
This solution determines  an $R$-submodule of the divided power   ring $\Gamma $ (or its completion) denoted by $I^{\perp} $ and called the inverse system of $I$ which contains the same information as in the original ideal.
Macaulay at the beginning of the 20th century  proved that the Artinian Gorenstein $k$-algebras are in correspondence with the cyclic $R$-submodules  of   $\Gamma $ where the elements of $R$ act as derivatives on $\Gamma$, see \cite{Ems78}, \cite{IK99}.
In the last twenty years several authors have applied this device to several problems, among others: Warings's problem,  \cite{Ger96},
n-factorial conjecture in combinatorics and geometry, \cite{Hai94},
the cactus rank, \cite{RS13},  the geometry of the punctual Hilbert scheme of Gorenstein schemes, \cite{IK99}, Kaplansky-Serre's  problem, \cite{RS}, classification up to analytic isomorphism of Artinian Gorenstein rings, \cite{ER12}.

\smallskip
The aim of this paper is 
to extend the well-known Macaulay's correspondence
 characterizing  the submodules of $\Gamma $  in one-to-one correspondence with Gorenstein d-dimensional $k$-algebras 
 (Theorem \ref{bijecGor}). These submodules are called $G_d$-admissible (Definition \ref{G-modules}) and 
in positive dimension they are not finitely generated. The $G_d$-admissible submodules of $\Gamma$ can be described in some coherent manner and we discuss effective methods for constructing Gorenstein $k$-algebras    with a particular emphasis to standard graded $k$-algebras (Theorem \ref{gradedGor}). In Section 4 several examples are given, in particular we propose a finite  procedure for constructing  Gorenstein graded $k$-algebras of given multiplicity or given Castelnuovo-Mumford regularity (Proposition \ref{finite}).  We discuss possible obstructions in the local case corresponding to non-algebraic curves. 
Our hope is that our results will be successfully applied to give new insights in the above mentioned applications and problems.

\medskip
The computations are performed in characteristic zero ($k= \mathbb Q$) by using the computer program system Singular, \cite{DGPS}, and the library \cite{E-InvSyst14}.

\section{Inverse System}

Let $V$ be a vector space of dimension $n$ over a field $k$ where, unless specifically stated otherwise,  $k $ is  a field of any characteristic.  Let $R= Sym_{\cdot}^k V = \oplus_{i \ge 0} Sym_i^k V $ be  the  standard graded polynomial ring in $n$ variables over $k$ and $\Gamma = D_{\bf \cdot}^k(V^*) = \oplus_{i\ge 0} D_i^k(V^*) = \oplus_{i\ge 0}  {\rm Hom}_k (R_i, k) $    be the graded $R$-module of graded $k$-linear homomorphisms from $R$ to $k.$ Through the paper  if  $V $ denotes the $k$-vector space $\langle z_1, \dots, z_n \rangle,  $ then we denote  by   $V^* = \langle Z_1, \dots, Z_n \rangle$   the  dual base and  $\Gamma = \Gamma(V^*) \simeq  k_{DP} [Z_1, \dots, Z_n] $ the divided power ring.   In particular   $\Gamma_j = \langle \{Z^{[L]} \  / \  |L|=j \} \rangle $ is the span of the dual generators to $z^{L}=z_1^{l_1}\cdots z^{l_n}$ where $L$ denotes the multi-index $L=(l_1,\dots, l_n)\in \mathbb N^n $ of length $|L|= \sum_i  l_i.$
If $L\in \mathbb Z^n$ then we set $ X^{[L]}=0$ if any component of $L$ is negative.   The monomials  $ Z^{[L]} $ are called divided power monomials (DP-monomials) and the elements $F= \sum_L b_L Z^{[L]} $ of $\Gamma $ the divided power polynomials (DP-polynomials).  

\vskip 2mm
We extend the above setting to the local case considering $R$ as the power series ring on $V. $  If $V $ denotes the $k$-vector space $\langle z_1, \dots, z_n \rangle, $ then $R=k[\![z_1,\dots, z_n]\!] $ will denote the formal power series ring  and   $\m=(z_1,\dots, z_n) $  denotes  the maximal ideal of $R. $    The injective hull $E_R(k) $  of $R$  is     isomorphic to the divided power ring (see \cite{Gab59}).  For detailed information   see  \cite{EisGTM}, \cite{Ems78}, \cite{IK99}, Appendix A. 

\vskip 2mm
 
We recall   that $\Gamma $ is a  $R$-module acting $R$ on  $\Gamma$ by {\it {contraction}}   as it follows.

\smallskip
\begin{definition} If $h =\sum_M a_M z^M \in R $ and $F= \sum_L b_L Z^{[L]} \in \Gamma, $ then the contraction of $F$ by $h$ is defined as
$$ h \circ F =  \sum_{M, L} a_M b_L Z^{[L-M]} $$
\end{definition}

The contraction is $Gl_n(k)$-equivariant.    If the characteristic of the field $k$ is zero, then there is a natural isomorphism of $R$-algebras between $(\Gamma , \circ) $ equipped with an internal product and the usual polynomial ring $P $  replacing the contraction with  the  partial derivatives.      In this paper we  do not consider the ring structure of $\Gamma, $ but  we  always consider  $\Gamma$ as $R$-module by  contraction and  $k$ will be a field of any characteristic.  

\medskip 
The contraction $\circ $ induces a exact pairing:
$$
\begin{array}{ cccc}
\langle\ , \ \rangle : & R  \times \Gamma  &\longrightarrow &   k  \\
                       &       (f , g) & \to  & ( f  \circ g ) (0)
\end{array}
$$
 
\medskip
If $I\subset R$ is an ideal of $R$ then  $(R/I)^\vee={\rm Hom}_R(R/I, \Gamma)$ is the $R$-submodule of $\Gamma$
$$
{I^{\perp}} =\{ g \in \Gamma \ |\  I \circ g = 0 \  \} =\{g\in \Gamma \ |\  \langle f, g \rangle = 0 \ \ \forall f  \in I \}.
$$
This submodule of $\Gamma $  is called {\it{Macaulay's inverse system of $I$.}}
If $I$ is a homogeneous ideal of a polynomial ring $R, $ then $I^{\perp}$ is homogenous (generated by forms in $\Gamma$ in the standard meaning) and
$ I^{\perp} = \oplus I_j ^{\perp} $ where $I_j ^{\perp} = \{ F \in \Gamma_j \mid  \ h \circ F=0 \ \ \mbox{ for \ all \ }  h \in I_j \}.$

\medskip

\noindent Given a $R$-submodule $W$ of $\Gamma, $ then the dual $W^\vee={\rm Hom}_R(W, \Gamma)$ is the ring $R/\ann_R(W)$ where
$$ \ann_R(W)= \{ f \in R \ \mid \ f \circ g= 0 \ \mbox{ for \ all \ } g \in W\}.
$$
Notice that $\ann_R (W) $ is an ideal of $R$.
Matlis duality assures that
$$
\ann_R(W)^{\perp}=W,  \;\; \ann_R(I^{\perp})=I.
$$
 
 If $W$ is generated by homogeneous DP-polynomials, then  $\ann_R(W) $ is a  homogeneous ideal of $R$.
  
Macaulay in \cite[IV]{mac16} proved a particular case of Matlis duality, called Macaulay's correspondence,
between the
  ideals $I\subseteq R  $ such that $R/I$ is an Artinian local ring and $R$-submodules  $W=I^{\perp}$ of $\Gamma $ which are finitely generated.
Macaulay's  correspondence is an effective method for computing Artinian rings, see  \cite{CI12}, Section 1, \cite{Iar94}, \cite{Ger96} and \cite{IK99}.

\medskip
If $(A, \maxn) $ is an Artinian local ring,  we denote by  $Soc(A) = 0 :_A \maxn$  the  socle of $A.$ Throughout this paper we denote by $s$ the {\it{socle degree}}  of $A$ (also called L\"{o}wey length),   that is the maximum integer $j$ such that $\maxn^j \neq 0.$ The  {\it{type}} of $A$ is $t(A) := \dim_{k} Soc(A)$;
 $A$ is an Artinian  Gorenstein ring if $t(A)=1$.
If $R/I$ is an  Artinian  local algebra   of  socle-degree  $s$ then $  {I^{\perp}} $ is generated by DP-polynomials  of degree $\le s  $  and $\dim_{k} (A)  (= \text{\ multiplicity\ of } A) = \dim_{k} I^{\perp}. $

\medskip
\noindent {\it{From Macaulay's  correspondence,   Artinian  Gorenstein $k$-algebras $A=R/I$ of socle degree $s$ correspond to cyclic $R$-submodules  of $\Gamma$ generated by a divided power polynomial $F\neq 0$ of degree $s$.}}

\medskip
We will denote by $\langle F \rangle_R $ the cyclic $R$-submodule of $\Gamma $ generated by the divided power polynomial $F.$ 

\medskip 
We can compute the Hilbert function of a graded or local $k$-algebra $A=R/I $ (not necessarily Artinian) in terms of its inverse system. The   Hilbert function of $A=R/I$ is by definition
$$
\HF_A(i) =   \dim_{k} \left(\frac{\maxn^i}{\maxn^{i+1}}\right)
$$
where $\maxn= \m/I $ is the maximal ideal of $A.$

We denote by $\Gamma_{\le i}$ (resp. $\Gamma_{< i}$, resp. $\Gamma_i$), $i\in \mathbb N$,  the $k$-vector space of DP-polynomials of $\Gamma$ of degree less or equal (resp. less, resp. equal ) to $i$, and we consider the following $k$-vector space
$$
 (I^{\perp})_i := {\frac{ I^{\perp} \cap \Gamma_{\le i} +  \Gamma_{< i}}{ \Gamma_{< i}}}. 
$$
Notice that if $I$ is an homogeneous ideal of the polynomial ring $R$, then  $(I^{\perp})_i = (I_i)^{\perp}. $  

\medskip
\begin{proposition}
\label{hf}
With the previous notation and for  all $i\ge 0$  
$$
\HF_{A}(i) = \dim_{k}   (I^{\perp})_i.
$$
\end{proposition}
\begin{proof}
Let's consider the following natural exact sequence of $R$-modules
$$
0 \longrightarrow
\frac{\maxn^i}{\maxn^{i+1}}
 \longrightarrow
 \frac{A}{\maxn^{i+1}}
  \longrightarrow
  \frac{A}{\maxn^{i}}
   \longrightarrow
   0.
$$
Dualizing this sequence we obtain
$$
0 \longrightarrow
(I+\max^{i})^\perp
 \longrightarrow
(I+\max^{i+1})^\perp
\longrightarrow
\left(\frac{\maxn^i}{\maxn^{i+1}}\right)^\vee
\longrightarrow
0
$$
so we get the following sequence of $k$-vector spaces:
$$
\left(\frac{\maxn^i}{\maxn^{i+1}}\right)^\vee \cong
\frac{(I+\max^{i+1})^\perp}{(I+\max^{i})^\perp} =
\frac{I^\perp \cap \Gamma_{\le i}}{I^\perp \cap \Gamma_{\le i-1}}\cong
{\frac{ I^{\perp} \cap \Gamma_{\le i} +  \Gamma_{< i}}{ \Gamma_{< i}}}.
$$
Then  the result follows since $\dim_{k} \left(\frac{\maxn^i}{\maxn^{i+1}}\right)^\vee = \dim_{k} \frac{\maxn^i}{\maxn^{i+1}} =\HF_A(i).$
\end{proof}

\medskip
\begin{example} {\rm{ Let $I=(xy, y^2-x^3) \subseteq R=k[[x,y]]$  and let $\Gamma=k_{DP}[X,Y].  $  It is easy to see that
$$ I^{\perp} =\langle X^{[3]}+ Y^{[2]}\rangle_R $$
and  $\langle X^{[3]}+ Y^{[2]}\rangle_R=\langle X^{[3]}+ Y^{[2]},  X^{[2]}, X,  Y, 1\rangle_k $ as $k$-vector space.
Hence by using Proposition \ref{hf}, one can compute the Hilbert series of $A= R/I$
$$\HS_{A}(z)= \sum_{i\ge 1} \HF_{A} (i) z^i = 1 + 2z +z^2+z^3. $$
 }}
\end{example}

\section{Structure of the inverse system}
 
In this section  $R$   denotes   the power series ring and $\m $ the maximal ideal.   
Let $I$ be an ideal of  $ R $ such  that $I\subset \max^2$ and  $A=R/I $ is Gorenstein of dimension $d \ge 1. $  Where specified $A=R/I  $ will be a standard graded algebra and in this case $R$ will be the polynomial ring and $\m$ the homogeneous maximal ideal. 

\vskip 2mm
We  assume that the ground field  $k$ is infinite. If $A$ is a standard graded $k$-algebra it is well known that we can pick   $\uz:=z_1,\dots,z_d   $  which are part of a basis of $R_1$ and they represent  a linear system of parameters for $R/I.$   If $A =R/I $ is a  local $k$-algebra we can pick   $\uz:=z_1,\dots,z_d \in \m \setminus \m^2 $  which are part of a minimal set of generators of $\m$ and their cosets represent a system of parameters for $R/I.$ 
In both cases we will say that  $\uz:=z_1,\dots,z_d $ is a {\it{regular linear sequence}} for $R/I.$ We remark that $z_1,\dots,z_d $ can be extended to a minimal system of generators of $\max,$ say $z_1,\dots,z_d, \dots, z_n  $ where $n=\dim R.$ If $ z_1,\dots  z_n$ is a minimal set of generators of $\max, $  we denote by  $Z_1, \dots, Z_n$   the corresponding dual basis  such that $z_i \circ Z_j =\delta_{ij}, $    hence $\Gamma =k_{DP}[Z_1,   \dots, Z_n]. $

\medskip
Assume
${{\uz}}:={ {z_1}},\dots,{ {z_d }} $   a regular linear sequence for $A. $ 
For every  $L=(l_1,\dots,l_d)\in \mathbb N^d $ we denote by  $\uz^L $ the sequence of pure powers
$z_1^{l_1},\dots ,z_d^{l_d}\in R$.  
Consider $L=(l_1, \dots, l_d) \in  \mathbb N_+^d$, we denote by
$$ \Gamma_{\uz^L} = (\uz^L)^{\perp} $$
the $R$-submodule of $\Gamma $ orthogonal to $\uz^L. $    Let $W = I^{\perp} $ be  the inverse system of $I$ in $\Gamma  $ and let
\begin{equation} \label{Wz}
W_{\uz^L} = W \cap  \Gamma_{\uz^L} = (I +  (\uz^L) )^{\perp}.
\end{equation}
 Since  $A/(\uz^L)A$ is an Artinian  Gorenstein local ring  for all $L\in \mathbb N_+^d$, see for instance \cite{BH97} Proposition 3.1.19(b),
then  $W_{\uz^L}$ is a non-zero cyclic  $R$-submodule of $\Gamma$ for all $L\in \mathbb N_+^d$. We are interested in  special generators of   $W_{\uz^L} $  strictly related to a given generator of  $W_{\uz}.$

We consider in $\mathbb N^d$, $d\le n$,  the componentwise ordering, i.e. given two multi-indexes
$L =(l_1\dots,l_d)$ and $M=(m_1, \dots,m_d)  \in \mathbb N^d$  then
$L\le_d M$ if and only if $l_i \le m_i $ for all $i=1,\dots, d$.
We recall that $|L|=l_1+\dots+l_d$ is the total degree of $L$.
If $ L \in \mathbb N^d, $  we denote by $\Gamma_{\le L}$, resp. $\Gamma_{< L}$,  the set of  elements of $\Gamma$ of multidegree less or equal (resp. less) than $L$  with respect to   $Z_1, \dots, Z_d. $ We remark that if $\uz=z_1, \dots, z_d,  $ then
\begin{equation} \label{L}
(\uz^L)^{\perp} = \Gamma_{<L}.
\end{equation}

\vskip 2mm
\begin{lemma}
\label{cut}
If  $L\in \mathbb N_+^d, $ then

\noindent
(i) $W_{\uz^L} / \m \circ W_{\uz^L} \cong  Soc (A/\uz^L A)^\vee $,

\noindent
(ii) $\bigcup_{L \in \mathbb N_+^d}  W_{\uz^L} = I^{\perp} $
\end{lemma}
\begin{proof} (i) see \cite{CI12}, Lemma 1.9 (ii). \par
\noindent  (ii) follows from (\ref{L}) and (\ref{Wz}) because  $$\bigcup_{L \in \mathbb N_+^d}  W_{\uz^L} = \bigcup_{L \in \mathbb N_+^d} (I^{\perp} \cap \Gamma_{<L}) = I^{\perp}. $$
\end{proof}

\medskip
 
For all $i=1,\dots,d$ we denote by ${\bf{\gamma_i}} = (0, \dots, 1, \dots 0) \in \mathbb N^d $ the   $i$-th  the coordinate vector.
We write $\md=(1,\dots,1)\in \mathbb N^d, $ more in general, for all positive integer  $t \in  \mathbb N$ we write
$${\bf{t}}_d =(t,\dots,t) \in  \mathbb N^d
$$

\medskip
The following result is a consequence of a modified Koszul complex on $R/I.$

\begin{proposition}
\label{koszul}
$(i)$
Assume $d=1$.
For all $l \ge 2$ there is an exact sequence of finitely generated $R$-submodules of $\Gamma$
$$
0
\longrightarrow
W_{z_1}
\longrightarrow
 W_{z_1^l}
 \stackrel{z_1\circ }{\longrightarrow}  W_{z_1^{l-1}}
 \longrightarrow 0.
 $$

\noindent
$(ii)$
Assume $d\ge 2. $
If $\uz=z_1,\dots, z_d, $ then
for all  $L\in \mathbb N^r$ such that $  L \ge  {\bf{2}}_d $ there is an  exact sequence  of finitely generated $R$-submodules of $\Gamma$

 $$
0
 \longrightarrow
 W_{\uz}
 \longrightarrow
W_{\uz^L}
\stackrel{ }{\longrightarrow}
\bigoplus_{k=1}^{d}  W_{\uz^{L-\gamma_k}}
\stackrel{ }{\longrightarrow}
\bigoplus_{1\le i<j\le d}  W_{\uz^{L-\gamma_i-\gamma_j}}
$$
\end{proposition}
\begin{proof}
$(i)$
Assume $d=1$.
Since $z_1$ is regular on $R/I, $ for all $l \ge 2, $ the following  sequence of Artinian Gorenstein rings is exact
$$
0
\longrightarrow \frac{R}{I+(z_1^{l-1})}
\stackrel{\cdot z_1}{\longrightarrow}  \frac{R}{I+(z_1^{l})}
\longrightarrow
\frac{R}{I+(z_1)}
\longrightarrow
0
$$
This induces  by duality the following exact sequence of finitely generated $R$-submodules of $\Gamma$
$$
0
\longrightarrow
W_{z_1}
\longrightarrow
 W_{z_1^l}
 \stackrel{z_1\circ }{\longrightarrow}  W_{z_1^{l-1}}
 \longrightarrow 0.
 $$

\noindent
$(ii)$
Assume $d\ge 2. $
If $\uz=z_1,\dots, z_d, $ then we prove that
for all  $L\in \mathbb N^d$ such that $  L \ge  {\bf{2}}_d $ the following   sequence of $R$-modules is exact
$$
\bigoplus_{1\le i<j\le d}  \frac{R}{I+(\uz^{L-\gamma_i-\gamma_j})}
\stackrel{\phi_L}{\longrightarrow}
\bigoplus_{k=1}^{d}  \frac{R}{I+(\uz^{L-\gamma_k})}
\stackrel{\varphi_L}{\longrightarrow}
 \frac{R}{I+(\uz^{L})}
 \longrightarrow
 \frac{R}{I+(\uz)}
 \longrightarrow
 0
$$
where: $\varphi_L(\overline{v_1},\dots,\overline{v_d})=\sum_{k=1}^d z_k  \overline{v_k}  $, and

\noindent
$\phi_L(\overline{v_{i,j}};1\le i<j\le d)=\sum_{1\le i<j\le r} (0,\dots,z_j,\dots,-z_i,\dots,0) \overline{v_{i,j}}$;

\noindent
for short we denote by $\overline {v } $  the class of an element $v\in R $ in the above different quotients.

Since  $  L \ge  {\bf{2}}_d $  for all $1\le i< j\le d$ we have
$L-\gamma_i-\gamma_j\in \mathbb N^d$.
It is easy to prove that $\varphi_L \phi_L=0$, so we have to prove that
$\ker(\varphi_L)\subset \img(\phi_L)$.
Given  $(\overline{v_1},\dots,\overline{v_d})\in \ker(\varphi_L)$ we have that
$\sum_{k=1}^d z_k v_k \in I+(\uz^L)$,
so there are $\lambda_1,\dots, \lambda_r\in R$ such that
$$
\sum_{k=1}^d z_k (v_k- \lambda_i z_k^{l_k-1}) \in I.
$$
Since $\uz$ is a regular sequence on $A=R/I$ we deduce that, modulo $I$,
$$
((v_k- \lambda_i z_k^{l_k-1})_{k=1,\dots,d}) \equiv \sum_{1\le i<j\le d}\mu_{i,j}(0,\dots,z_j^{ },\dots,-z_i^{ },\dots,0)
$$
for some $\mu_{i,j}\in R$, $1\le i<j\le d$.
From this we deduce that
 $(\overline{v_1},\dots,\overline{v_d})\in \img(\phi_L)$.

Now the exact sequence of Artinian Gorenstein rings  induces by Matlis duality  the following exact sequence  of $\Gamma$-modules
$$
0
 \longrightarrow
 W_{\uz}
 \longrightarrow
W_{\uz^L}
\stackrel{\varphi_L^*}{\longrightarrow}
\bigoplus_{k=1}^{d}  W_{\uz^{L-\gamma_k}}
\stackrel{\phi_L^*}{\longrightarrow}
\bigoplus_{1\le i<j\le d}  W_{\uz^{L-\gamma_i-\gamma_j}}
$$
with
$\varphi_L^*(v)=(z_1\circ v,\dots, z_d\circ v)$
and
$\phi_L^*(v_{1},\dots, v_d)= (z_j\circ v_i-z_i\circ v_j;1\le i<j\le d)$. This proves $(ii).$
\end{proof}

\bigskip
We recall the following basic fact that will be useful in the following.

 \begin{lemma}
\label{back}
Let $(R, \m, k) $ be a local ring and let $f:M\longrightarrow N$ be an epimorphism between two non-zero cyclic
$R$-modules.
Let $m\in M$ be an element such that $f(m)$ is a generator of $N$.
Then $m$ is a generator of $M$.
\end{lemma}
\begin{proof}
We remark that $f$ induces an isomorphism    of
one-dimensional $k$-vector spaces $\bar f:M/\max M\longrightarrow N/\max N.$ 
Since the coset of $f(m)$ in $N/\max N$ is non-zero the coset of $m$ in
$M/\max M$ is non-zero.
Hence $m$ is a generator of $M$.
\end{proof}

\medskip

We remark that if $L = \md, $ then $  W_{\uz^\md} =W_{\uz} = (I+(z_1,\dots, z_d))^{\perp}$.  Then   $W_{\uz} $ is a non-zero cyclic $R$-submodule of $\Gamma $   and denote by $H_{\md}  $  a generator:
$$ W_{\md} = \langle H_{\md}\rangle.$$
In particular $W_{\uz} $ is the dual of the Artinian reduction $R/I +(\uz). $ It is clear that $H_{\md}$ depends from the regular sequence $\uz$ we consider.
Our goal is  to lift the generator $H_{\md}$ of $W_{\md}$ to a suitable generator $H_L  $ of $W_{\uz^L}= (I +  (\uz^L) )^{\perp}$, for all $L=(l_1,\dots,l_d)\in \mathbb N_+^d $.

\bigskip
\begin{proposition}
\label{lifting}
For all $L=(l_1,\dots,l_d)\in \mathbb N_+^d  $  and  for all $i=1,\dots, d$ such that $l_i\ge 2, $ let  $ H_{L-\gamma_i}$  be a generator  of $ W_{\uz^{L-\gamma_i}}.$ There exists  a  generator $H_L$ of $W_L$ satisfying
$$
z_i \circ H_L=H_{L-\gamma_i}
$$
 for all $i=1,\dots, d$ such that $l_i\ge 2$.
\end{proposition}
\begin{proof}
For every $L\in \mathbb N^d_+$ we define $|L|_+$ as the number of positions $i\in \{1,\dots,d\}$
such that $l_i\ge 2$.
We proceed by recurrence  on the pair $(|L|_+, |L|-(d-|L|_+))\in \{1,\dots, d\}\times \mathbb N$.
Notice that
$|L|-(d-|L|_+)\ge 2 |L|_+$.

Assume that $|L|_+=1$.
After a  permutation, we may assume that
$L=(l,1,\dots,1)$ with $l\ge 2$.
Consider the ideal $J=I+(z_2,\dots,z_d)$;
from \propref{koszul} $(i)$ we get an exact sequence
$$
0
\longrightarrow
W_{\md}
\longrightarrow
W_L=(J+(z_1^l))^\perp
\stackrel{z_1\circ}{\longrightarrow}
W_{L-\gamma_1}=(J+(z_1^{l-1}))^\perp
\longrightarrow
0.
$$
Notice that
$|L-\gamma_1|_+\le 1$ and that if $|L-\gamma_1|_+= 1$ then
$|L-\gamma_1|-(d-1)= |L|-(d-1)-1$.
 Hence by induction  we know that there exists $H_{L}\in W_L$ satisfying
 $z_1 \circ H_L=H_{L-\gamma_1}$ and
$W_{L-\gamma_1}=\langle H_{L-\gamma_1}\rangle$.
\lemref{back} applied to the epimorphism
$$
W_L=(J+(z_1^l))^\perp
\stackrel{x_1\circ}{\longrightarrow}
W_{L-\gamma_1}=(J+(z_1^{l-1}))^\perp
\longrightarrow
0
$$
with $m=H_L$ gives that
 $H_L$ is a generator of $W_L$.

We may assume that $r=|L|_+\ge 2$.
After a  permutation we may assume that
$L=(l_1,\dots, l_r,1\dots, 1)$ with $l_i\ge 2$ for $i=1,\dots, r$.
We set $\uz'=z_1,\dots,z_r$ and  $L'=(l_1,\dots, l_r)\in \mathbb N_+^r$.
Consider the ideal $J=I+(z_{r+1},\dots,z_d)$;
from \propref{koszul} $(ii)$ we get an exact sequence
$$
0\longrightarrow
T_{\mr}
\longrightarrow
T_{L'}
\stackrel{\varphi_{L'}^{*}}{\longrightarrow}
\bigoplus_{k=1}^{r}  T_{L'-\gamma_k}
\stackrel{\phi_{L'}^*}{\longrightarrow}
\bigoplus_{1\le i<j\le r}  T_{L'-\gamma_i-\gamma_j}
$$
where
$T_{\mr}=W_{\md}$,
$T_{L'}=(J+\uz'^{L'})^\perp=(I+ \uz^L)^\perp=W_L$,
$T_{L'-\gamma_k}=(J+\uz'^{L'-\gamma_k})^\perp=(I+ \uz^{L-\gamma_k})^\perp=W_{{L- \gamma_k}}$
and
$T_{L'-\gamma_i-\gamma_j}=(J+\uz'^{L'-\gamma_i-\gamma_j})^\perp=(I+ \uz^{L-\gamma_i-\gamma j})^\perp=W_{L-\gamma_i-\gamma j}$.
 Hence, by induction,  we know that for all $k=1, \dots, r$ there exists $H_{L-\gamma_k}\in \Gamma$ such that
 $W_{L-\gamma_k}=\langle H_{L-\gamma_k}\rangle$
 and
  $z_i \circ H_{L-\gamma_k}=H_{L-\gamma_k-\gamma_i}$ for all $i\in\{ 1,\dots, r\}, i\neq k$.
  From this we deduce that
  $$
  (H_{L-\gamma_1},\dots,H_{L-\gamma_r} )\in \ker(\phi_{L'}^*),
  $$
from the above exact sequence there exists  $H_L\in W_L$ such that
$z_k\circ H_L=H_{L-\gamma_k}$ for all $k=1,\dots,r$.
The same argument as before proves that $H_L$ is a generator of $W_L$.
\end{proof}

\bigskip
\begin{remark}
\label{integral}
With the above notation,   given two  DP-polynomials $H, G\in \Gamma, $ we say that $G$ is a primitive of $H$ with respect to $z_1 \in R  $ if
$z_1\circ G=H$.  
From the definition of $\circ, $ we will get   
$$G= Z_1  H + C$$
for some $C\in \Gamma$ such that $z_1 \circ C=0.  $   Remark that $Z_1  H $  denotes the usual multiplication in a polynomial ring and we do not use the internal multiplication in $\Gamma$ as DP-polynomials. 
Hence in  Proposition \ref{lifting},  we will say that $H_{L}$    is a primitive  of $H_{L-\gamma_i}$
with respect to $z_i$  for all $i=1,\dots,d.  $
\end{remark}

\bigskip
We prove now the main result of this paper which  is an extension to the $d$-dimensional case of Macaulay's Inverse System correspondence. We give a complete description of the  $R$-submodules of $\Gamma$ whose annihilator is a $d$-dimensional Gorenstein local ring. In the Artinian case they are cyclic generated by a polynomial of $\Gamma, $ in positive dimension the dual modules are not finitely generated and further conditions  will be required.

\bigskip
\begin{definition}
\label{G-modules} Let $d\le n$ be a positive integer.   An   $R$-submodule $M  $  of  $ \Gamma$   is called $G_d$-admissible if it admits a system of  generators $\{H_L\}_{ L\in \mathbb N_+^d} $ in $ \Gamma=k_{DP}[Z_1, \dots, Z_n]$  satisfying  the following conditions
\begin{enumerate}
\item[(1)] for all $L\in \mathbb N_+^d$ and for all $i=1,\dots,d$
$$
z_i \circ H_L =
\left\{
\begin{array}{ll}
 H_{L-\gamma_i}   & \text{ if }  L-\gamma_i >0   \ \\
0 & \text{ otherwise.}
\end{array}
\right.
$$
\item[(2)] $\ann_R( H_L )\circ H_{L+\gamma_i}=\langle H_{L-(l_i-1)\gamma_i}\rangle$ for all $i=1,\dots,d$ and
$L=(l_1,\cdots,l_d)\in \mathbb N_+^d$.
\end{enumerate}

\medskip

\noindent
If this is the case, we also say that $M$ is
$G_d$-admissible with respect to the elements   $z_1, \dots, z_d \in R$.
\end{definition}

\begin{remark}
\label{nonzero}
Given a $G_d$-admissible set  $\{H_L\}_{ L\in \mathbb N_+^d} $ in $\Gamma, $
  the condition
  $H_{\md}=0$  is equivalent to the vanishing of the $R$-module  $M=\langle H_L, L\in \mathbb N_+^d\rangle$.
  In fact, assume that $H_{\md}=0$.
  We proceed by induction on $t=|L|$ where $L\in \mathbb N_{+}^d$.
  If $t=d$ then $L=\md$ and $H_L=0$ by hypothesis.
Assume that $H_L=0$  for all $L$ with $|L|\le t$.
We only need to prove that $H_{L+\gamma_i}=0$ for all $i=1,\cdots,n$.
Since $H_L=0$ we get that $\ann_R(H_L)=R$.
From the condition $(ii)$ of the above definition we get
$R\circ H_{L+\gamma_i}=\langle H_{L-(l_i-1)\gamma_i}\rangle=0$,
so $H_{L+\gamma_i}=0$.
The converse is trivial.
\end{remark}

\vskip 2mm
 
 \begin{theorem}
\label{bijecGor}
Let $(R, \m) $ be the power series ring and let $d\le \dim R$ be a positive integer.   There is a one-to-one correspondence $\mathcal C$ between the following sets:

\vskip 2mm
\noindent
(i)
 $d$-dimensional Gorenstein quotients of  $R$,

\noindent
(ii) non-zero $G_d$-admissible $R$-submodules $M=\langle H_L, L\in \mathbb N_+^d\rangle$ of $\Gamma$.

\vskip 2mm
\noindent
In particular, given an ideal $I\subset R$ with $A= R/I$ satisfying $(i)$, then
 $$\mathcal C(A)= I^{\perp} = \langle H_L,  L\in \mathbb N_+^d\rangle \subset \Gamma   \ \ {\text{with}}\ \
  \langle H_L\rangle=(I+(\uz^L))^\perp $$ is $G_d$-admissible with respect to a regular linear sequence $\uz=z_1,\dots, z_d $ for $R/I.$ 
 Conversely, given a  $R$-submodule $M$ of $\Gamma $ satisfying  (ii), then $$\mathcal C^{-1}(M)=R/I  \ \ {\text{with}}\ \
 I= \ann_R(M) = \bigcap_{L\in \mathbb N_+^d}\ann_R(H_L)$$ is a $d$-dimensional Gorenstein ring. 

\end{theorem}
 \begin{proof}
Let $A=R/I $ be a quotient of $R$  satisfying $(i)$ and consider $\uz=z_1, \dots, z_d$ in $R$ a linear regular sequence modulo $I.$ Let $z_1, \dots, z_d, \dots, z_n$ be a minimal system of generators of $\max$ and let $Z_1, \dots, Z_n$ be the dual base.
Let $H_{\md}$ be a generator of $W_{\md}=(I+(\uz))^{\perp}$ in $\Gamma=k_{DP}[Z_1, \dots, Z_n].$
Since $d\ge 1$ we have $H_{\md}\neq 0$.
By  \propref{lifting} there exist  elements  $H_L$, $L\in\mathbb N_+^d$ in $\Gamma$ such that $W_{\uz^L}= (I+ (\uz^L))^{\perp} =\langle H_L\rangle$  and  by Lemma \ref{cut}  $M=I ^{\perp} = \langle H_L, L\in\mathbb N_+^d\rangle$ satisfies  Definition \ref{G-modules} (1).

Since $W_{\uz^{L+\gamma_i}}=\langle H_{L+\gamma_i}\rangle$ we have
$(I+(\uz^{L+\gamma_i}))\circ H_{L+\gamma_i}=0$.
In particular $I\circ H_{L+\gamma_i}=z_j^{l_j}\circ H_{L+\gamma_i}=0$
for all $j\in\{1,\dots,d\}$ and $j\neq i$.
Hence we get
$$
\ann_R(H_L)\circ H_{L+\gamma_i}
=(I+(\uz^{L}))\circ  H_{L+\gamma_i}
=(z_i^{l_i})\circ  H_{L+\gamma_i}
=\langle H_{L-(l_i-1)\gamma_i}\rangle.
$$
It follows that  $M=\langle H_L, L\in\mathbb N_+^d\rangle$  is a $R$-submodule of $\Gamma$ which is $G_d$-admissible and we set  $\mathcal C(A)=M. $
Since $H_{\md}\neq 0$ we have that $M\neq 0$.

\medskip
Conversely, let $M\neq 0$ be a $R$-submodule of $\Gamma= k_{DP}[Z_1, \dots, Z_n]$ which is $G_d$-admissible. Hence  $M$ admits a system of  generators $\{H_L\}_{ L\in \mathbb N_+^d} $ satisfying  the following conditions
\begin{enumerate}
\item[(1)]  For all $L\in \mathbb N_+^d$ and for all $i=1,\dots,d$
$$
z_i \circ H_L =
\left\{
\begin{array}{ll}
 H_{L-\gamma_i}   & \text{ if }  L-\gamma_i >0   \ \\
0 & \text{ otherwise.}
\end{array}
\right.
$$
\item[(2)] $\ann_R(H_L)\circ H_{L+\gamma_i}=\langle H_{L-(l_i-1)\gamma_i}\rangle$ for all $i=1,\dots,d$ and $L\in \mathbb N_+^d$,
\end{enumerate}
\noindent
and $H_{\md }\neq 0$, Remark \ref{nonzero}.

For all $L\in \mathbb N_+^d$
we set  $W_L:=\langle H_L\rangle$ and $I_L:=\ann_R(W_L)$.
We define the following  ideal of $R$
$$I:=\bigcap_{L\in \mathbb N_+^d} I_L $$
and we prove that $R/I$ is Gorenstein of dimension $d.$

\bigskip
\noindent
{\bf Claim 1:} For all $L\in \mathbb N_+^d$  it holds
$I_L\subset I_{L+\md} + (\uz^L)$.

\noindent
{\bf Proof:}
Notice that it is enough to prove that $I_L\subset I_{L+\gamma_i} + (z_i^{l_i})$
for $i=1,\dots,d$.
In fact, assume that $I_L\subset I_{L+\gamma_i} + (z_i^{l_i})$
for all $i=1,\dots,d$.
Then
\begin{multline*}
I\subset I_{L+\gamma_1} + (z_1^{l_1})\subset I_{L+\gamma_1+\gamma_2} + (z_1^{l_1},z_2^{l_2})\subset
\dots \\
\dots
\subset
 I_{L+\gamma_1+\dots+\gamma_d} + (z_1^{l_1},\dots,z_d^{l_d})=I_L\subset I_{L+\md} + (\uz^L).
\end{multline*}

We  prove now that
$I_L\subset I_{L+\gamma_1} + (z_1^{l_1})$.
Given $\beta\in I_L=\ann_R(H_L)$ by $(2)$ there is $\gamma\in R$ such that
$$
\beta\circ H_{L+\gamma_1}=\gamma \circ H_{L-(l_1-1)\gamma_1}= \gamma \circ (z_1^{l_1} \circ H_{L+\gamma_1}).
$$
From this identity we deduce
$\beta- \gamma z_1^{l_1}\in \ann_R(H_{L+\gamma_1})=I_{L+\gamma_1}$,
so
$\beta\in I_{L+\gamma_1} + (z_1^{l_1})$.

\bigskip
\noindent
{\bf Claim 2:} For all $L\in \mathbb N_+^d$  it holds
$I_L= I + (\uz^L)$.

\noindent
{\bf Proof:}
By (1) we get
 $(\uz^L)\circ H_L=0$, hence  $(\uz^L)\subset I_L$.
 Since $I\subset I_L$ we get the inclusion
 $I+(\uz^L)\subset I_L$.

Now we prove that $I_L\subset I+(\uz^L)$.
 Given  $\beta_L\in I_L$, by {\bf Claim 1} there are $\beta_{L+\md}\in I_{L+\md}$ and
 $\omega^0 \in (\uz^L)$ such that
 $$
 \beta_L= \beta_{L+\md}+\omega^0.
 $$
 Since $\beta_{L+\md}\in I_{L+\md}$, by {\bf Claim 1} there are $\beta_{L+ {\bf{2}}_d }\in I_{L+{\bf{2}}_d}$ and
 $\omega^1 \in (\uz^{L+\md})$ such that
 $$
 \beta_{L+\md}= \beta_{L+{\bf{2}}_d}+\omega^1,
 $$
 so $\beta_L= \beta_{L+{\bf{2}}_d} +\omega^0+\omega^1.$
 By recurrence  there are sequences $\{\beta_{L+t \md}\}_{t\ge 0}$ and
 $\{\omega^t\}_{t\ge 0}$ such that $\beta_{L+{\bf{t}}_d}\in I_{L+{\bf{t}}_d}$,  $\omega^t\in (\uz^{L+ {\bf{t}}_d})$.
 For all
 $t\ge 0$ it holds
 $$
 \beta_L=\beta_{L+{\bf{t}}_d} + \sum_{i=0}^{t} \omega^t.
 $$

\noindent
Since $\omega^t\in (\uz^{L+{\bf{t}}_d})$ for all $t\ge 0$ we get that there exists
 $\overline{\omega}=\sum_{t\ge 0} \omega^t \in k[\![z_1,\cdots , z_n]\!]$.
 Hence there exists $\overline{\beta}=\lim_{t\rightarrow \infty} \beta_{L+{\bf{t}}_d} \in k[\![z_1,\cdots , z_n]\!]$, so
 $$
 \beta_L=\overline{\beta}+\overline{\omega}.
 $$
 Notice that $\overline{\omega}\in (\uz^L)$ and that
 $\overline{\beta}\in \bigcap_{t\ge 0} I_{L+{\bf{t}}_d}=I$.
 From this we get that $\beta_L\in I+(\uz^L)$.

 If $R= k[z_1,\cdots , z_n]$ then $\overline{\beta} \in I\subset R$.
 Since $\beta_L\in R$ we get that $\overline{\omega}=\sum_{t\ge 0} \omega^t \in R=k[z_1,\cdots , z_n]$.

 \bigskip
\noindent
{\bf Claim 3:} $\uz$ is a regular sequence of  $R/I $ and $\dim (R/I)= d$.

\noindent
{\bf Proof:}
Since $W_{\md}=\langle H_{\md}\rangle = I_{\md}^{\perp}$, by {\bf Claim 2}
$I_{\md}^{\perp}=(I+(\uz))^{\perp}$.
By Remark \ref{nonzero} we get  that $H_{\md}\neq 0$ since $M\neq 0$.
Hence
$R/I+(\uz)$ is Artinian and then  $\dim (R/I)\le d$.
Next we  prove that $\uz$ is a regular sequence modulo $I  $ and hence $\dim (R/I)= d$.

First we prove that $z_1$ is a non-zero divisor of $A=R/I$.
By $(1)$ the derivation by $z_1$ defines  an epimorphism  of $R$-modules
$$
W_L=\langle H_L\rangle \stackrel{z_1\circ}{\longrightarrow}
W_L=\langle H_{L-\gamma_1}\rangle
\longrightarrow
0
$$
for all $L\ge_d {\bf{2}}_d$.
This sequence induces an exact sequence of $R$-modules
$$
0
\longrightarrow
\frac{R}{I+(\uz^{L-\gamma_1})}
\stackrel{. z_1}{\longrightarrow}
\frac{R}{I+(\uz^{L})}.
$$
Let $a\in R$ such that $z_1 a\in I$.
Since $z_1 a\in I+ (\uz^{L})$ we deduce that
$ a\in I+ (\uz^{L-\gamma_1})$ for all $L\ge_d {\bf{2}}_d$,
and we conclude that $a\in I$.

\noindent
Assume that $z_1,\dots, z_r$, $r< d$, is a regular sequence of $R/I$.
Given  $L'=(l_{r+1},\dots, l_d)\in\mathbb N^{d-r}_+$ such that $l_i\ge 2$, for all $i=r+1,\dots,d$, we write $L=(1,\dots,1,l_{r+1},\dots, l_d)\in\mathbb N^{d}_+$.
By $(1)$ the derivation by $z_{r+1}$ defines  an epimorphism  of $R$-modules
$$
W_L=\langle H_L\rangle \stackrel{z_{r+1}\circ}{\longrightarrow}
W_L=\langle H_{L-\gamma_{r+1}}\rangle
\longrightarrow
0.
$$
This sequence induces an exact sequence of $R$-modules
$$
0
\longrightarrow
\frac{R}{I+(z_1,\dots,z_r)+(z_{r+1}^{l_r-1},\dots,z_d^{l_d})}
\stackrel{. z_{r+1}}{\longrightarrow}
\frac{R}{I+(z_1,\dots,z_r)+(z_{r+1}^{l_r},\dots,z_d^{l_d})}.
$$
Let $a\in R$ such that $z_{r+1} a\in I+(z_1,\dots,z_r)$.
Since $z_{r+1} a\in I+(z_1,\dots,z_r)+(z_{r+1}^{l_r},\dots,z_d^{l_d})$ we deduce that
$ a\in I+(z_1,\dots,z_r)+(z_{r+1}^{l_r-1},\dots,z_d^{l_d})$ for all $L'\ge_r {\bf{2}}_d$.
Hence $a\in I+(z_1,\dots,z_r), $ as wanted.

  \bigskip
\noindent
{\bf Claim 4:} $R/I $ is Gorenstein

\noindent
{\bf Proof:}
$(I+(\uz))^\perp$ is cyclic,  so $R/I+(\uz)$ is Gorenstein.
Since $R/I$ is Cohen-Macaulay by Claim 3,  we have that $R/I$ is a $d$-dimensional Gorenstein ring.

\bigskip

 We define
 $\mathcal C'(M)=R/I$ where $I=\bigcap_{L\in \mathbb N_+^d} I_L$ and we prove that $\mathcal C$ and $\mathcal C'$  are inverse each other.
 Let $A=R/I$ be
a $d$-dimensional Gorenstein rings $A=R/I$ such that $z_1,\dots,z_d$  is a regular sequence of $A$.
Then
$$
\mathcal C' \mathcal C (A)=R/J
$$
where
$$
J
=\bigcap_{L\in \mathbb N_+^d}  \ann_R((I+(\uz^L))^\perp)
=\bigcap_{L\in \mathbb N_+^d}  I+(\uz^L)=I,
$$
because $\ann_R((I+(\uz^L))^\perp)=I+(\uz^L)$.
Hence $\mathcal C' \mathcal C $ is the identity map in the set of rings satisfying  $(i)$.

Let $M=\langle H_L; L\in \mathbb N_+^d\rangle$  be an $R$-module satisfying  $(ii)$.
Then $ \mathcal C' (M)=R/I$ where
$I=\bigcap_{ L\in \mathbb N_+^d} \ann_R(H_L)$.
By {\bf Claim 2} we know $I+(\uz^L)=\ann_R(H_L)$ so
$\mathcal C \mathcal C' (M)=M.$
\end{proof}

\medskip

\begin{remark}
\label{grad}Theorem \ref{bijecGor} can be also applied to standard graded quotients $R/I$  of $R$ where $R$ is the polynomial ring (and $I$ is an homogeneous ideal). In this case the  dual $R$-submodule $M$    of $\Gamma $ will be generated by homogeneous DP-polynomials $H_L. $  Hence the correspondence will be between $d$-dimensional Gorenstein graded $k$-algebras  and $G_d$-admissible homogeneous $R$-submodules  $M$ of $\Gamma.  $  
\end{remark}

\begin{remark}
\label{details}
Let $A=R/I $ be a $d$-dimensional Gorenstein quotient of $R$  and let $\uz= z_1, \dots, z_d \in \m\setminus \m^2  $ be a linear regular sequence  for   $R/I. $ By Theorem \ref{bijecGor}, let $M=\langle H_L, L\in \mathbb N_+^d\rangle$ be   the dual module. Then by Macaulay's correspondence,  the  socle degree of the Artinian reduction $A/(\uz)A $ coincides with $\deg(H_{\md})$.
On the other hand we have the following inequality on the multiplicity of $A: $
$$
e_0(A)\le \length_R(A/(\uz)A)=\dim_{k}(\langle H_{\md}\rangle).
$$
The equality holds in the graded case or in the local case if $\uz$ is a superficial sequence  of $A. $   If $A$ is a standard graded $k$-algebra which is Gorenstein (hence Cohen-Macaulay), then the multiplicity and the Castelnuovo Mumford regularity of $A$ coincide with those of any Artinian reduction, in particular $A/(\uz)A.$ Hence $e(A) = \dim_{k}(\langle H_{\md}\rangle) $ and $\reg(A) = \deg H_{\md}.$
\end{remark}

\medskip
According to the previous remark, in the graded case important geometric information such as the multiplicity, the arithmetic  genus, or more in general,  the Hilbert polynomial of the Gorenstein $k$-algebra depend only on the choice of $H_{\md}, $ the first step in the construction of a  $G_d$-admissible dual module, see Example \ref{fot}.
We can state the following result which refines Theorem \ref{bijecGor} in the case of graded $k$-algebras.

\begin{theorem}
\label{gradedGor}
 Let $R $ be the polynomial ring and let $d\le \dim R$ be a positive integer. There is a one-to-one correspondence $\mathcal C$ between the following sets:

\noindent
(i) $d$-dimensional Gorenstein standard graded $k$-algebras  $A=R/I$  of multiplicity $e=e(A) $ (resp. Castelnuovo-Mumford regularity $r= \reg(A)$)

\noindent
(ii) non-zero $G_d$-admissible homogeneous $R$-submodules $M=\langle H_L, L\in \mathbb N_+^d\rangle$ of $\Gamma$  such that $\dim_{k} \langle H_{\md}\rangle = e$ (resp. $\deg H_{\md} =r$)
\end{theorem}

\medskip
The following remarks will be useful in the effective construction of a $G_d$-admissible $R$-submodule of $\Gamma. $  

\begin{remark}
\label{intersection} Condition (2) in Definition \ref{G-modules} can be replaced by the following:

\vskip 2mm
(2) {  $\langle H_L \rangle \cap k[Z_1, \dots, {\widehat {Z_i } }, \dots, Z_n] \subseteq \langle  H_{L-(l_i-1)\gamma_i} \rangle$ for all $L\in \mathbb N_+^d$ and for all $i=1,\dots,d. $}

\vskip 2mm \noindent {\bf Proof:} Let  $\{H_L\}_{L\in \mathbb N_{+}^d}$ be in $\Gamma$ satisfying conditions (1) and (2) of Definition \ref{G-modules}, we prove condition (2) above. Let $\beta_i \in   \langle H_L \rangle \cap  k[Z_1, \dots, {\widehat {Z_i } }, \dots, Z_n],$ then there exists $\alpha_i \in R, $ such that $\beta_i = \alpha_i \circ H_L \in   k[Z_1, \dots, {\widehat {Z_i } }, \dots, Z_n] $ and hence $z_i \circ \beta_i =0.$ Now $\alpha_i \circ H_{L-\gamma_i}=\alpha_i \circ (z_i \circ H_L) = z_i \circ (\alpha_i \circ H_L)= z_i \circ \beta_i =0, $ hence $\alpha_i \in \ann (H_{L-\gamma_i}) $ and by assumption $ \beta_i \in \langle  H_{L-(l_i-1)\gamma_i} \rangle. $

Conversely assume (2) above. Let $\alpha_i \in \ann (H_{L-\gamma_i})$  and we prove $\alpha_i \circ H_L \subseteq \langle H_{L-(l_i-1)\gamma_i}\rangle.$ We have $z_i \circ (\alpha_i \circ H_L) = \alpha_i \circ (z_i \circ H_L) = \alpha_i \circ H_{L-\gamma_i} =0, $ hence $ \alpha_i \circ H_L \in  k[Z_1, \dots, {\widehat {Z_i } }, \dots, Z_n].$ It follows $\alpha_i \circ H_L \in  \langle H_L \rangle \cap  k[Z_1, \dots, {\widehat {Z_i } }, \dots, Z_n]  \subseteq \langle H_{L-(l_i-1)\gamma_i} \rangle, $ as required.

\end{remark}
\vskip 2mm

\begin{remark}
\label{simplify}
It is easy to see that   condition $(2)$ of  Definition \ref{G-modules}  is equivalent to
$$
H_{L+\gamma_i}\in (\ann_R( H_{L-(l_i-1)\gamma_i} ) \ann_R(H_L ))^{\perp}.
$$
It is worth mentioning that any $G_d$-admissible set $\{H_L\}_{L\in \mathbb N_{+}^d}$  with $M=\langle H_{L}; L\in \mathbb N_{+}^d\rangle$, satisfies

\noindent
$(i)$
$M=\langle H_{{\bf{t}}_d}; t\ge n_0\rangle$
for any integer $n_0\ge 1$,

\noindent
$(ii)$
if  $L\in \mathbb N_+^d, $ then
$H_{L+\md}\in (\ann_R( H_{\md} )\ann_R( H_L )^d)^{\perp}$.

\noindent
{\bf Proof:}
$(i)$
Notice that for all $L\in \mathbb N_+^d$ there is $t\in \mathbb N$ such that
$L\le_d {\bf{t}}_d$,
so
$H_L$ is determined by $H_{{\bf{t}}_d}$ because $H_L=\uz^{{\bf{t}}_d-L}\circ H_{{\bf{t}}_d}$, see \thmref{bijecGor}.

\noindent
$(ii)$
Since $L+\md=(L+\md-\gamma_1)+\gamma_1$,  from \thmref{bijecGor} $(ii)$ we get
that
$$
\ann_R(H_L)\circ H_{L+\md}=\langle H_{L+\md-l_1 \gamma_1}\rangle.
$$
Since $\sum_{i=1}^d l_i \gamma_i=L$, by recurrence we deduce that
$$
\ann_R(H_L)^d\circ H_{L+\md}=\langle H_{L+\md- L}\rangle=\langle H_{\md}\rangle.
$$
From this identity we get $(ii)$.
\end{remark}

\begin{remark}
\label{shape}
Let $M= \langle H_L; L\in \mathbb N_+^d \rangle $ be a non-zero $R$-submodule of $\Gamma $ which is $G_d$-admissible with respect to a linear regular sequence $z_1, \dots, z_d \in R$.
 Let   $Z_1, \dots, Z_d  $ the corresponding dual elements in $\Gamma_1.$
 As a  consequence of  Definition \ref{G-modules},  for $t\ge 1  $  and in accordance  with Remark \ref{integral}, we can    write
$$ H_{({\bf{t+1}})_d}= Z_1\cdots Z_d  H_{{\bf{t}}_d} + C_{t+1}=\sum_{i=0}^{t} Z_1^i \cdots Z_d^i \, C_{t+1-i}$$ where
$(z_1\dots z_d)\circ C_i=0$  for all $i=1, \dots, t+1. $
Notice that  by the above remark the diagonal elements $ H_{({\bf{t+1}})_d} $ can describe the module $M=\langle H_L, L\in {\mathbb N}^d_+ \rangle.$
\end{remark}

\section{Examples and effective constructions}

Many interesting questions arise from Theorem \ref{bijecGor}, but   the most challenging aim is to deepen the  effective aspects.
 We are interested in  the construction   of $R$-submodules $M$ of $\Gamma$ which are $G_d$-admissible and to their corresponding  Gorenstein $k$-algebras.  We hope that the following examples can suggest interesting refinements of our main result.  
 
 \medskip
 The following is a first  (trivial) example.

\medskip
If $J$ is an ideal in $k[\! [z_{d+1},\dots,z_n]\!], $   we say that an ideal $I\subset R=k[\! [z_{1},\dots,z_d, \dots, z_n]\!] $ is a cone with respect to  $J$  if $I=JR.$
In the following result denote $S= k[\! [z_{d+1},\dots,z_n]\!] $ and let $Q= k_{DP}[Z_{d+1}, \dots, Z_n]  $ be the corresponding dual.

\begin{proposition} \label{cone}  Given $H \in  Q, $   consider  the following $R$-submodule   of $ \Gamma=k_{DP}[z_{1},\dots, z_n]$
$$M=  \langle H_L=Z_1^{l_1-1}\dots Z_d^{l_d-1} H \mid    L=(l_1,\dots,l_d)\in \mathbb N^d_+ \rangle.$$
 Then  $R/\ann_R(M) $ is a $d$-dimensional Gorenstein graded $k$-algebra  and $\ann_R(M) $ is a cone with respect to  $J =\ann_{S} (H). $
\end{proposition}
\begin{proof} We prove that $M$ is $G_d$-admissible proving that $M$ satisfies  Definition \ref{G-modules} with respect  to the sequence $z_1, \dots, z_d. $
It is easy to show that $M$ satisfies condition (1)   setting  $H_{\md}=H. $
We prove now that $M$    satisfies also condition (2) of  Definition \ref{G-modules}, that is $\ann_R( H_L ) \circ H_{L+\gamma_i}\subseteq \langle H_{L-(l_i-1)\gamma_i}\rangle$ for all $i=1,\dots,d$ and $L\in \mathbb N_+^d$.
First of all we observe that, since  $H\in k[Z_{d+1},\dots,Z_n], $  it is easy to prove that
$$
\ann_R\langle Z_1^{l_1 }\dots Z_d^{l_d } H\rangle=
\ann_R\langle Z_1^{l_1  }\dots Z_d^{l_d } \rangle+ \ann_S (H)R = (z_1^{l_1 }, \dots , z_d^{l_d }) + \ann_S (H)R
$$
Hence
\begin{multline*}
  \ann_R(H_L )\circ H_{L+\gamma_i} = ((z_1^{l_1 }, \dots , z_d^{l_d }) + \ann_S (H)R) \circ H_{L+\gamma_i}=   \\
  =(z_1^{l_1 }, \dots , z_d^{l_d }) \circ H_{L+\gamma_i}  \  \subseteq \langle H_{L-(l_i-1)\gamma_i}\rangle.
  \end{multline*}
 From \thmref{bijecGor} we know that there exists a  $d$-dimensional Gorenstein local ring $A=R/I$ such that
 $I^\perp= \langle  H_L; L\in\mathbb N_+^d\rangle$ and
 $ I= \ann_R(M) = \bigcap_{L\in \mathbb N_+^d}\ann_R(H_L).$
Hence
$$
I =  \bigcap_{L\in \mathbb N_+^d} ((z_1^{l_1 }, \dots , z_d^{l_d }) + \ann_S (H)R) = \ann_S (H)R
$$
is a cone with respect to  $J =\ann_{S} (H). $
\end{proof}

\medskip

 We are interested now in more significant classes of Gorenstein $d$-dimensional $k$-algebras.     The dual module of a Gorenstein ring of positive dimension is not finitely generated and  we would be interested in effective methods determining  Gorenstein rings in a finite numbers of steps.    This aim motives the following setting.

 \medskip
Let $t_0 \in \mathbb N_+, $  we say that  a  family   $\mathcal H= \{H_L;L\in \mathbb N_+^d, |L|\le t_0\}$  of polynomials of $\Gamma $ is {\it{admissible }} if the elements  $H_L$ satisfy  conditions (1) and (2) of Definition \ref{G-modules} up to $L$ such that $|L|\le t_0.$ 



We recall that, starting from a polynomial $H_{\md}, $ Remark \ref{intersection} and Remark \ref{simplify} give  inductive procedures for constructing an admissible set
$\mathcal H=\{H_L;L\in \mathbb N_+^d, |L|\le t_0\}$ with respect to a sequence of linear elements $\uz=z_1,\cdots,z_d. $

\begin{proposition} \label{finite}
Let  $\mathcal H=\{H_L;L\in \mathbb N_+^d, |L|\le t_0\}$ be an admissible set of homogenous polynomials with respect to a
linear sequence $\uz= z_1,\cdots,z_d$ with $t_0\ge (r+2)d$ where  $r = \deg H_{\md} $.
Assume there exists a graded Gorenstein   $k$-algebra  $A=R/ I$ such that
 $(I+ (\uz^L))^{\perp} =\langle H_L\rangle$ for all $|L|\le t_0$.
Then $$ I= \ann_R(H_{({\bf{r+2}})_d})_{\le r+1}R.$$
\end{proposition}
\begin{proof}
It well known that  the Castelnuovo-Mumford $\reg(A)$
 regularity of the Cohen-Macaulay ring $A=R/I$ coincides with regularity of the Artinian reduction
 $R/\ann_R( H_{\md})$, and hence with its  socle degree.
Hence we get that
 $$
 \reg(A)=s(R/\ann_R(H_{\md}))=\deg H_{\md}=r
 $$
Now   the maximum degree of a minimal system of generators of $I$ is at most $\reg(A) +1$.
 From the identity  $\ann_R(H_{({\bf{r+2}})_d})=I+(\uz^{r+2})$
we get the claim.
\end{proof}

\begin{remark}
\label{degree}
  Notice that the last result provide a test for checking if an admissible set of polynomials can be lifted to a full
  $G_d$-admissible family in the graded case.
  We only need to check if the ideal
  $ I= \ann_R(H_{({\bf{r+2}})_d})_{\le r+1}R$
  is Gorenstein.
  
  It is worth mentioning that if we know the maximum degree  of the generators of $I$ (which coincides with those of the $\ann_R(H_{\bf 1_d}$), say $b, $ then
  $$ I= \ann_R(H_{({\bf{b+1}})_d})_{\le b}R$$ 
\end{remark}

 \medskip
\noindent We  show  now some effective constructions of  Gorenstein   $k$-algebras by using the  above results and remarks. 
In the following example we construct a  $1$-dimensional Gorenstein graded $k$-algebra of codimension two.

 \medskip
 \begin{example}
   Consider $R=k[x,y,z] $ and $ H_1= Y^{[3]}-Z^{[3]} \in \Gamma=k_{DP}[X,Y,Z].$  Notice that $$\ann_R(H_1) = (x,yz, y^{ 3}+z^{3}) $$

\noindent
Our aim is to construct an ideal  $I \subseteq R=k[x,y,z] $ such that $A=R/I$ is Gorenstein and $R/I+(x) = R/\ann_R(H_1). $   We have $r=\deg H_1= 3, $ but from the above  equality we deduce the maximum degree of the generators of the ideal $I $   is three. 
Hence, by Proposition \ref{finite} and Remark \ref{degree},  a $1$-dimensional lift  of the Artinian reduction $R/\ann_R(H_1) $ it is univocally determined by $H_4$ in an admissible set $  \mathcal H= \{ H_1, \dots, H_4\}. $  In particular
$$I= \ann_R(H_4)_{\le 3} R.
$$
By using Singular  one can verify that the following set is admissible:

\noindent
$\mathcal H= \{ H_1, H_2 =XH_1  +YZ^{[3]},  H_3 = XH_2   -Y^{[2]}Z^{[3]}, H_4 = XH_3 +Y^{[3]}Z^{[3]}  -4Z^{[6]} \}  $
and hence
  $$ I= (yz+xz, y^3 +z^3 -xy^2+x^2y -x^3). $$
Notice that $A=R/I$ is a one-dimensional Gorenstein ring of multiplicity $6$, $x$ is a linear regular  element of $A$ and $R/\ann_R(H_1)$ is an Artinian reduction of $A$.
\end{example}

In the following example we construct a  two-dimensional Gorenstein $k$-algebra of codimension three.

\medskip
\begin{example} 
\label{fot}
 
Consider $R=k[x,y,z,t,w] $ and $$ H_{1,1}= X^{[2]}+Y^{[2]}+XZ \in \Gamma=k_{DP}[X,Y,Z,T,W].$$
Our aim is to construct an ideal  $I \subseteq R=k[x,y,z, t, w] $ such $A=R/I$ is Gorenstein and $R/I+(t,w) = R/\ann_R(H_{1,1}). $ We have $r=\deg H_{1,1}=2 .$
By Proposition \ref{finite},   a $2$-dimensional lift  of the Artinian reduction $B= R/\ann_R(H_{1,1}) $ it is univocally determined by $H_{4,4}$ in an admissible set $  \mathcal H= \{ H_{1,1}, \dots, H_{4,4}\}. $  In particular
$I= \ann_R(H_{4,4})_{\le 3} R.$

Since the Hilbert function of $B$ is $\{1,3,1\}, $ by \cite[Theorem B]{Sch80} we know that $B$ is the quotient of $k[x,y,z]$ by an ideal minimally generated
by  $5$ forms of degree two.
Hence $I$ is minimally generated by $5$ forms of degree two and then
$$
I= \ann_R(H_{4,4})_{\le 2} R=\ann_R(H_{3,3})_{\le 2} R.
$$

By using Singular  one can verify that the following collection of polynomials  forms an admissible set:

{\tiny

\medskip
\noindent
$ H_{1,1}= X^{[2]}+Y^{[2]}+XZ$,

\medskip
\noindent
$H_{2,2}=2X^{[4]}+XY^{[3]}+X^{[2]}YZ+2Z^{[4]}-X^{[3]}T+Y^{[2]}ZT+XZ^{[2]}T+X^{[3]}W-X^{[2]}YW-Z^{[3]}W-TW H_{1,1}$,

\medskip
\noindent
$H_{3,3}=X^{[5]}Y+X^{[2]}Y^{[4]}-X^{[5]}Z+X^{[3]}Y^{[2]}Z+Y^{[5]}Z+X^{[4]}Z^{[2]}+XY^{[3]}Z^{[2]}+Y^{[4]}Z^{[2]}+X^{[2]}YZ^{[3]}+3Y^{[3]}Z^{[3]}-X^{[2]}Z^{[4]}+3XYZ^{[4]}-3Z^{[6]}
-3X^{[5]}T+Y^{[5]}T+XY^{[3]}ZT+X^{[2]}YZ^{[2]}T+3Z^{[5]}T
+X^{[4]}T^{[2]}+Y^{[2]}Z^{[2]}T^{[2]}+XZ^{[3]}T^{[2]}-X^{[5]}W+X^{[4]}YW-X^{[3]}Y^2W-Y^{[5]}W-2X^{[4]}ZW-XY^{[3]}ZW-Y^{[4]}ZW-X^{[2]}YZ^{[2]}W-2Y^{[3]}Z^{[2]}W+X^{[2]}Z^{[3]}W-2XYZ^{[3]}W+2Z^{[5]}W+X^{[3]}YW^{[2]}+XY^{[3]}W^{[2]}+Y^{[4]}W^{[4]}-X^{[3]}ZW^{[2]}+X^{[2]}YZW^{[2]}+Y^{[3]}ZW^{[2]}-X^{[2]}Z^{[2]}W^{[2]}+XYZ^{[2]}W^{[2]}-Z^{[4]}W^{[2]}-TW H_{2,2},$}

\noindent
We are ready now to compute $ \ann_R(H_{3,3})_{\le 2} R, $ hence
  $$ I= (z^2-xt+zt+zw+tw,yz-t^2+yw,-y^2+xz+t^2,-xy+zt+t^2,x^2-xz-yt+zt-xw+tw). $$
Notice that $A=R/I$ is a two-dimensional Gorenstein ring of multiplicity $5$, $\{t,w\}$ is a linear regular sequence  of $A$ and $B$ is an Artinian reduction of $A$.
The projective scheme $C$  defined by $A$ is a non-singular arithmetically Gorenstein elliptic curve of $\mathbb P_k^4$.
Notice that the above generators of $I$ are the Pfaffians of the skew matrix
$$
\left(
  \begin{array}{ccccc}
    0&-x+t&-t&x&-y\\
  x-t&0&x&-y&z+t\\
  t&-x&0&z+w&0\\
  -x&y&-z-w&0&-t\\
  y&-z-t&0&t&0\\
      \end{array}
\right)
$$
Since the Hilbert function of $B$ is $\{1,3,1\}$
we get
that the arithmetic genus of $C$, that coincides with its geometric genus,  is $e_1(B)-e_0(B)+1=5-5+1=1$,
where $e_0(B), e_1(B)$ are the   Hilbert coefficients of   $B$, see \cite{RV-Book}.

\end{example}

\medskip

In the following example we construct a  $1$-dimensional Gorenstein $k$-algebra of codimension four.

\medskip
\begin{example} 
\label{fot2}
Consider $R=k[x,y,z,t,v] $ and $ H_{1}= X^{[2]}+Y^{[2]}+Z^{[2]}+T^{[2]} \in \Gamma=k_{DP}[X,Y,Z,T,V].$
We have $r=\deg H_{1}=2 .$
Hence a $1$-dimensional lift  of the Artinian reduction $R/\ann_R(H_{1}) $ it is univocally determined by $H_{4}$ in an admissible set $  \mathcal H= \{ H_{1}, \dots, H_{4}\}. $
In particular
$$I= \ann_R(H_{4})_{\le 3} R.
$$
One can verify that the following collection of polynomials  forms an admissible set:

\medskip
\noindent
{\tiny
$H_{1}= X^{[2]}+Y^{[2]}+Z^{[2]}+T^{[2]}$,

\medskip
\noindent
$H_{2}=V H_1 + X^{[3]}+Z^{[3]}$,

\medskip
\noindent
$H_{3}=V H_2 + X^{[4]}+Z^{[4]}$,

\medskip
\noindent
$H_{4}=V H_3+ X^{[5]}+Z^{[5]}$.}

\noindent

\bigskip
Hence
  $$ I= (x^2-z^2-xv+zv, xy,y^2-z^2+zv,xz,yz,z^2-t^2-zv,xt,yt,zt). $$
Notice that $A=R/I$ is a one-dimensional Gorenstein ring of multiplicity $6$, $v$ is a linear regular  element  of $A$ and $R/\ann_R(H_{1})$ is a minimal reduction of $A$ with Hilbert function $\{1,4,1\}$.
In this case $A$ is non-reduced and a minimal prime decomposition is
$$
I=(x,y,z-v,t)\cap (x-v, y, z, t)\cap (t^3,t^2+zv,zt,z^2,yt,yz,y^2+zv,x-z)
$$
with minimal primes
$$
(x,y,z-v,t),  (x-v, y, z, t), (x,y,z,t);
$$
the minimal graded $R$-free resolution of $A$ is
$$
0
\longrightarrow
R(-6)
\longrightarrow
R(-4)^9
\longrightarrow
R(-3)^{16}
\longrightarrow
R(-2)^9
\longrightarrow
R
\longrightarrow
A=R/I
\longrightarrow
0.
$$

\end{example}

\medskip
We present now examples in the local (non-homogeneous) case.  Possible obstacles to a finite procedure could come in particular when $R/I$ is not algebraic.       The ring  $A=R/I$ is algebraic if there is an ideal $J\subset T=k[z_1,\dots,z_n]_{(z_1,\dots,z_n)} $
such that $A$ is analytically isomorphic to $R/JR$, completion of $T/J$ with respect to
the $(z_1,\dots,z_n)$-adic topology.  If the singularity defined by $A=R/I$ is isolated,  then $A$ is algebraic.
It was proved by Samuel for hypersurfaces, \cite{Sam56}, and,  in general, by Artin in
\cite[Theorem 3.8]{Art69b}. But there are singularities of normal surfaces in $\mathbb C^3$ which  are not algebraic, \cite[Section 14, Example 14.2]{Whi65a}.
Notice that the ideal defining the above singularity is principal, in particular is Gorenstein.

\medskip
The following example suggests that in the quasi-homogeneous case,   Proposition \ref{finite} could  be still true.

\begin{example}
Consider $R=k[[x,y,z]] $ and we construct  a non-homogeneous ideal $I$ in $R$ such that $R/I $ is Gorenstein of dimension $1 $ and multiplicity $5.$  By Theorem \ref{bijecGor} we should  exhibit  a $R$-submodule $M$ of $\Gamma=k_{DP}[X,Y,Z] $ which is $G_1$-admissible.  Remark \ref{details} suggests to consider a polynomial $H_1$ such that $\dim_k <H_1>=5. $ Let $H_1= Z^{[2]} +Y^{[3]} $ (non-homogeneous, but quasi-homogeneous).  In this case $\deg H_1 =3 $ and $H_1$    is a quasi-homogeneous polynomial.  One can verify that $$ \mathcal H=\{ H_1, H_2=XH_1,  H_3=X^{2}H_1, H_4= X^{[3]}H_1+ Y^{[4]} Z +Y Z^{[3]} , H_5=X H_4\}  $$
 is  an admissible set. In this case is still true that $$ I =\ann_R(H_5)_{\le 4}R =(yz-x^3, z^2-y^3). $$
 Notice that the ideal $I$ is  the defining ideal in $R$ of the semigroup ring $k[[t^5,t^6,t^9]].$  In particular $A= R/I $ is a domain.   \end{example}

 \vskip 2mm
 The above example suggests the interesting problem   to characterize   the generators of the dual module in Theorem \ref{bijecGor} of a  Gorenstein domain.
 \vskip 2mm
 Next example shows an example where $A$ is a local Gorenstein $k$-algebra, but the corresponding associated graded ring is not longer Gorenstein.
 \vskip 2mm

 \begin{example}
 Consider $R=k[[x,y,z,t,u,v]] $ and we construct an  ideal $I$ in $R$ such that $R/I $ is Gorenstein of dimension $2.$ By Theorem \ref{bijecGor} we exhibit  a $R$-submodule $M$ of $\Gamma=k_{DP}[X,Y,Z,T,U,V] $ which is $G_2$-admissible.  Let $H=Z^{[5]}+T^{[4]}+U^{[3]}+W^{[3]}+ZTUV.    $ One can verify that $$ M=\langle H, F=XYH+U^{[2]}T-WTZ,  X^{i}Y^{j}F, \ i,j \in \mathbb N \rangle $$
 is $G_2$-admissible. Then $$ \ann_R(M) = (z^4-tuw, t^2w,z^2w,t^2u, t^3-zuw, zt^2, z^2t, w^2-ztu, u^2-tu^2-ztw-xyzuw). $$ In particular   $A= R/\ann_R(M) $ is a Gorenstein local ring of dimension $2 $  and of codimension $4.$ Notice that $gr_{\maxn}(A) $ is not Gorenstein because the second difference of the Hilbert function (computed by using Proposition \ref{hf})  is not symmetric.
 \end{example}

We end  this paper with an  example showing that \propref{finite} cannot be extended to the local case without a suitable modification.

\begin{example}
For all $n\ge 2$ we consider the one-dimensional local ring $A_n=k[\![x,y]\!]/(f_n)$ with
$f_n=y^2-x^n. $ Notice that, for  all $n\ge 1,$  $A_n$ is algebraic and  Gorenstein of multiplicity $e(A_n)=2.$
 
On the other hand
$A_n/(x)=k[y]/(y^2) $ is an Artinian reduction of $A_n$, so $H_1=Y$ and hence $\deg H_1=1.$
If $n \ge 4, $  we cannot recover the ideal $(f_n)$ after $\deg H_1 +2 $ steps as \propref{finite} could suggest. 
\end{example}


\providecommand{\bysame}{\leavevmode\hbox to3em{\hrulefill}\thinspace}
\providecommand{\MR}{\relax\ifhmode\unskip\space\fi MR }
\providecommand{\MRhref}[2]{%
  \href{http://www.ams.org/mathscinet-getitem?mr=#1}{#2}
}
\providecommand{\href}[2]{#2}

\end{document}